\documentclass[a4paper,11pt]{article}
\usepackage{hyperref,tikz}
\usepackage{amsthm,amsmath,amssymb,amsfonts}
\usepackage[numbers, sort & compress]{natbib}

\def\pref#1{(\ref{#1})}

\makeatletter
\def\p@enumii{}
\makeatother

\newtheorem{thm}{Theorem}[section]
\newtheorem{lemma}[thm]{Lemma}
\newtheorem{proposition}[thm]{Proposition}

\theoremstyle{definition}

\newtheorem{example}[thm]{Example}

\theoremstyle{remark}

\numberwithin{equation}{section}

\newcommand{\chr}{\mathrm{char}\,}

\newcommand{\G}{\mathcal{G}}
\newcommand{\inv}{\mathrm{Inv}}

\sloppypar

\begin{document}
\title{Toroidal involutory Cayley graphs}

\author{Hamide Keshavarzi$^1$, Babak Amini$^{2}$, Afshin Amini$^3$ and Shahin Rahimi$^4$\\ 
\it\small Department of Mathematics, College of Sciences, Shiraz University, \\
\it\small 71457-44776, Shiraz, Iran\\
\it\small $^1$E-mail: hamide2326@hafez.shirazu.ac.ir\\
\it\small $^2$E-mail: bamini@shirazu.ac.ir\\
\it\small $^3$E-mail: aamini@shirazu.ac.ir\\
\it\small $^4$E-mail: shahin.rahimi.math@gmail.com}
\date{}

\maketitle
\begin{abstract}
Suppose that $R$ is a finite commutative ring with identity. The involutory Cayley graph $\G(R)$ of $R$ is the graph whose vertices are the elements of $R$, and two distinct vertices $x$ and $y$ are adjacent if and only if $(x-y)^2=1$. In this paper, we classify all rings $R$ for which $\G(R)$ is a toroidal graph, that is, a graph that can be embedded on a torus.
\end{abstract}


\section{Introduction}

Throughout this paper, all rings are commutative, finite and with a nonzero identity. Also, $R$ always denotes a ring. In addition, all considered graphs are simple and undirected.

In algebraic graph theory, Cayley graphs are a basic and extremely symmetric class of graphs. Their structural and combinatorial properties become especially interesting when constructed over rings, exposing deeper connections between algebra and graph theory, even though they have been extensively studied in the general setting of arbitrary groups. See, for example, \cite{unit1,unit2,akhtar,unit3,unit4,unit5,unit6,qunit1,qunit3,zunit1,zunit2,inv}.

Recall that an element $u$ of a ring is called \textit{involutory} or an \textit{involution}, if $u=u^{-1}$. Here, we consider the \textit{involutory Cayley graph} $\G(R)$ of a ring $R$, which is the Cayley graph of the abelian group $(R,+)$ with respect to the set $\inv(R)$ of involutory elements of $R$. In other words, the vertices of $\G(R)$ are the elements of $R$, and two vertices $x$ and $y$ are adjacent (denoted by $x\sim y$) when $(x-y)^2=1$. It can be seen that in $\G(R)$, each vertex $x$ is adjacent to $x+u$ for all $u\in \inv(R)$, implying that $\G(R)$ is a $|\inv(R)|$-regular graph with no isolated vertices. 

We studied the graph $\G(R)$ in \cite{inv}, identifying important graph parameters like its girth, chromatic number, edge chromatic number, clique number, and independence number. We also gave a classification of rings $R$ for which $\G(R)$ is either planar or self-complementary. 
In this paper, we classify all rings $R$ for which $\G(R)$ admits a toroidal embedding, extending our investigation.

Recall that a \textit{toroidal graph} is a graph that can be embedded on a torus, that is, its vertices and edges can be drawn on a torus without crossings except at their common endpoints. Any graph that can be embedded in a plane can also be embedded in a torus, so every planar graph is also a toroidal graph. A toroidal graph that cannot be embedded in a plane is said to have genus 1. More generally, the \textit{(orientable) genus} $g(G)$ of a graph $G$ is the smallest integer $n$ such that the graph can be drawn without crossing itself on a sphere with $n$ handles.

It is well-known that if $R$ is a finite commutative ring, then $R \simeq R_1\times \cdots\times R_t$, where each $R_i$ is a finite local ring, and this decomposition is unique up to a permutation of the factors (see \cite[Theorem 8.7]{ati}).  Moreover,
if $R$ is a finite commutative local ring with maximal ideal $M$, then by considering the chain $R\supseteq M \supseteq M^2 \supseteq \cdots$, it can be seen that  $|R|$, $|M|$ and $\chr(R)$ are all powers of some prime number $p$ (see, for example \cite[Proposition 2.1]{akhtar}). Furthermore, in such a ring $R$, the element $p$ is nilpotent, and any prime number $q\ne p$ is a unit.

Recall that the \textit{direct product} (also known as the conjunction product or Kronecker product) of two graphs $G$ and $H$, denoted by $G\otimes H$, has the vertex set $V(G)\times V(H)$, where two vertices $(g_1,h_1)$ and $(g_2,h_2)$ are adjacent in $G\otimes H$ if and only if $g_1$ is adjacent to $g_2$ in $G$ and $h_1$ is adjacent to $h_2$ in $H$. For a comprehensive background on the genus of graphs, see the book \cite{white}; for a discussion of some well-known graphs and related results, we refer to the survey \cite{sur}. Furthermore, let $C_n$ denote a cycle graph of length $n$. The genus of the direct product $C_m\otimes C_n$ is equal to $2$ if both $m$ and $n$ are even integers, otherwise it is $1$ (see \cite[Corollary 3]{conj}). 


\section{Toroidal involutory Cayley graphs}

In this section, we first recall some essential results from \cite{inv}. Then, we classify all rings $R$ for which $\G(R)$ is of genus $1$. To begin, we recall the following theorem which classifies planar involutory Cayley graphs, completing  the genus zero case. It is noteworthy that for a ring $R$, the graph $\G(R)$ is planar if and only if $\G(R)$ is $1$-regular or $2$-regular (see \cite[Proposition 6.2]{inv}).

\begin{thm}\label{planar}\cite[Theorem 6.8]{inv}
	The involutory Cayley graph of a ring $R$ is planar if and only if $R\simeq S$, $R\simeq T$, or $R\simeq S\times T$, where $S$ is a direct product of fields of characteristic $2$ and $T$ is either a local ring of odd characteristic or it is isomorphic to $\mathbb{Z}_4$, $\mathbb{Z}_2[x]/\langle x^2\rangle$,  $\mathbb{Z}_2[x]/\langle x^3 \rangle$, or $\mathbb{Z}_4[x]/\langle 2x, x^2-2\rangle$.
\end{thm}

Now, we turn our attention to the genus one case. To proceed with our work, we require the following preparatory result from \cite{inv}.

\begin{thm}\label{imp}
For any (finite commutative) ring $R$, the following statements hold.
\begin{enumerate}
\item \label{imp 1} \cite[Remark 2.2]{inv} The graph $\G(R)$ is $2^t$-regular for some nonnegative integer $t$.

\item \label{imp 2} \cite[Proposition 2.3]{inv} If $R$ is a local ring of odd size, then $\G(R)$ is a disjoint union of $|R| / \chr(R)$ copies of cycles of length $\chr(R)$.

\item \label{imp 3} \cite[Theorem 3.1]{inv} The graph $\G(R)$ is bipartite if and only if $|R|$ is an even integer.

\item \label{imp 4} \cite[Note 2.8]{inv} If $R=R_1\times \cdots \times R_t$ is a decomposition of $R$ into local rings, then $\G(R)\simeq \G(R_1)\otimes \cdots \otimes \G(R_t)$. In particular, if $\G(R_i)$ is a $k_i$-regular graph for each $i$, then $\G(R)$ is a $k$-regular graph with $k=k_1\cdots k_t$.
\end{enumerate}
\end{thm}

Using the above theorem, we present a characterization of connected involutory Cayley graphs.

\begin{proposition}\label{conn}
Let $R=R_1\times \cdots \times R_t$, where each $R_i$ is a local ring. Then, the graph $\G(R)$ is connected if and only if at most one of the $\G(R_i)$'s is a connected graph of even size and each remaining $R_i$ is isomorphic to $\mathbb{Z}_{p_i^{n_i}}$ for some odd prime number $p_i$ and positive integer $n_i$.
\end{proposition}
\begin{proof}
Since $R=R_1\times \cdots \times R_t$, by Theorem \ref{imp} \pref{imp 4}, we infer that $\G(R)\simeq\G(R_1)\otimes \cdots \otimes \G(R_t)$. According to \cite[Theorem 1]{connec}, we conclude that $\G(R)$ is connected if and only if $\G(R_i)$ are connected and at most one of them is bipartite. Consequently, by Theorem \ref{imp} \pref{imp 3}, at most one of the $R_i$'s (so $\G(R_i)$'s) is of even size and the others are of odd size. In addition, if $\G(R_i)$ is a connected graph of odd size, by Theorem \ref{imp} \pref{imp 2}, we infer that $|R_i|=\chr(R_i)$, implying $R_i\simeq \mathbb{Z}_{p_i^{n_i}}$ for some odd prime number $p_i$ and positive integer $n_i$. Hence, the result follows.
\end{proof}

In the following lemma, we present some necessary conditions for a ring $R$ to satisfy when $\G(R)$ is of genus $1$.

\begin{lemma}\label{nec}
If $R$ is a ring for which $\G(R)$ is of genus $1$, then  $\G(R)$ is connected and $4$-regular.
\end{lemma}
\begin{proof}
Since $\G(R)$ is of genus $1$, by \cite[Proposition 2.1]{wick}, its minimum degree is at most $6$. Taking Theorem \ref{imp} \pref{imp 1} into account, we infer that $\G(R)$ is $1$, $2$, or $4$-regular. As $1$-regular and $2$-regular graphs are clearly planar (that is, of genus $0$), we conclude that $\G(R)$ is a $4$-regular graph.

Moreover, it is known that the genus of a graph is the sum of the genera of its connected components (see \cite[p. 55]{white}). Furthermore, it is easily seen that in a Cayley graph, all connected components are isomorphic. Consequently, a Cayley graph of genus $1$ is connected, and hence so is $\G(R)$, as desired.
\end{proof}

In Theorem \ref{main}, we shall see that the converse of the above lemma is also true. Therefore, in the following lemmas, we aim to classify all rings $R$ whose involutory Cayley graphs are connected and $4$-regular. To this end, we begin with the nonlocal rings.

\begin{lemma}\label{nonl}
Let $R$ be a nonlocal ring. The graph $\G(R)$ is connected and $4$-regular if and only if $R$ is isomorphic to 
$$\mathbb{Z}_{p^n}\times \mathbb{Z}_{q^m} , \quad \mathbb{Z}_{p^n}\times \mathbb{Z}_{q^m}\times \mathbb{Z}_2 , \quad\mathbb{Z}_{p^n}\times \mathbb{Z}_4 \quad \text{or}\quad \mathbb{Z}_{p^n}\times \mathbb{Z}_2[x]/\langle x^2 \rangle $$
for some odd prime numbers $p,q$ and positive integers $n,m$.
\end{lemma}
\begin{proof}
($\Leftarrow$) It follows from Theorem \ref{imp} and Proposition \ref{conn}.

($\Rightarrow$) First, assume that $R$ is a product of two local rings. 
 If $|R|$ is odd,  Proposition \ref{conn} implies that $R=\mathbb{Z}_{p^n}\times \mathbb{Z}_{q^m}$ for some odd prime numbers $p,q$ and positive integers $m,n$. Otherwise, $|R|$ is even, and in this case, by Proposition \ref{conn}, we conclude that $R=\mathbb{Z}_{p^n}\times S$, where $p$ is an odd prime, $n$ a positive integer and $\G(S)$ a $2$-regular connected graph of even size (Note that $\G(\mathbb{Z}_{p^n})$ is $2$-regular by Theorem \ref{imp} \pref{imp 2}). Thus, $\G(S)$ is a cycle graph and by \cite[Corollary 2.7]{inv}, as $S$ is local, one infers that $S$ is isomorphic to either $\mathbb{Z}_4$ or $\mathbb{Z}_2[x]/\langle x^2 \rangle$, from which the desired result follows.
 
 Now, we assume that $R$ is a product of three local rings. If $|R|$ is odd, utilizing Proposition \ref{conn}, we infer that $R=\mathbb{Z}_{p_1^{n_1}}\times \mathbb{Z}_{p_2^{n_2}} \times \mathbb{Z}_{p_3^{n_3}}$ for some odd prime numbers $p_i$ and positive integers $n_i$. However,  Theorem \ref{imp} shows that $\G(R)$ would then be $8$-regular, a contradiction. Hence, $|R|$ must be even. Again, applying Proposition \ref{conn}, we deduce  that $R=\mathbb{Z}_{p^n}\times \mathbb{Z}_{q^m}\times S$, where $p,q$ are odd prime numbers, $n,m$ are positive integers and $S$ is a ring of even size such that $\G(S)$ is connected. Since $\G(\mathbb{Z}_{p^n})$ and $\G(\mathbb{Z}_{q^m})$ are $2$-regular by Theorem \ref{imp} \pref{imp 2}, then $\G(S)$ must be  $1$-regular. Connectivity forces $\G(S)\simeq K_2$, so $S\simeq \mathbb{Z}_2$, which gives the desired result
 
 Furthermore, if $R=R_1\times \cdots \times R_t$ is the decomposition of $R$ into local rings and $t\geq 4$, by Proposition \ref{conn}, at least three of $R_i$'s, say $R_1$, $R_2$ and $R_3$ are of odd size. Since $\G(R_1)$, $\G(R_2)$ and $\G(R_3)$ are $2$-regular graphs (by Theorem \ref{imp} \pref{imp 2}), we conclude that $\G(R)$ is not $4$-regular, which is a contradiction. Hence, the result follows.
\end{proof}

Now, we deal with local rings in three lemmas. We begin with the case where the underlying ring is of  characteristic $2$.

\begin{lemma}\label{ch2}
For a local ring $R$ of characteristic $2$, the graph $\G(R)$ is connected and $4$-regular if and only if $R$ is isomorphic to $\mathbb{Z}_2[x,y]/\langle x^2,xy,y^2 \rangle$.
\end{lemma}
\begin{proof}
($\Rightarrow$)
Let $M$ be the maximal ideal of $R$. Assuming $u\in \inv(R)$, we obtain $(u+M)^2=1+M$ in the field $R/M$ of characteristic $2$, implying that $u\in u+M = 1+M$. Since  $\G(R)$ is a $4$-regular graph, we can assume that 
$$\inv(R)=\{1,1+x,1+y,1+z\},$$
where $x,y,z\in M\setminus \{0\}$. Since $1=(1+x)^2=1+x^2$, we infer that $x^2=0$ and similarly, $y^2=z^2=0$. So, if $m\in M$ and $m^2=0$, then we can conclude that $m\in \{0,x,y,z\}$.

Now, note that $(x+y)^2=0$, so $x+y \in \{0,x,y,z\}$, from which we infer that $z=x+y$. We also have $(xy)^2=0$, which implies that $xy=0$, since if $xy=z$, then $z=xy=x(x+z)=x(x+xy)=x^2(1+y)=0$, which is impossible and the two other cases also lead to a contradiction.
By a similar argument, we obtain that $xz=yz=0$. It is evident that each member of $A=\{0,x,y,z\}$ is adjacent to every member of $1+A=\{1,1+x,1+y,1+z\}$.  Since $\G(R)$ is a connected $4$-regular graph, we infer that $R=A\cup (1+A)$. Consequently, $|R|=8$ and $M=A$. Clearly, $M^2=0$, $\G(R)\simeq K_{4,4}$ and $$R=\{a+bx+cy\mid a,b,c = 0,1\}.$$
Therefore, we deduce that $R$ is isomorphic to $\mathbb{Z}_2[x,y]/ \langle x^2,xy,y^2\rangle$.

($\Leftarrow$) Note that for $R=\mathbb{Z}_2[x,y]/ \langle x^2,xy,y^2\rangle $, we have
$\G(R)\simeq K_{4,4}$, which is a connected $4$-regular graph.
\end{proof}

The next lemma deals with local rings of characterstic $4$.

\begin{lemma}\label{ch4}
Let $R$ be a local ring of characteristic $4$. The graph $\G(R)$ is connected and $4$-regular if and only if $R$ is isomorphic to $\mathbb{Z}_4[x]/\langle x^2,2x \rangle $.
\end{lemma}
\begin{proof}
($\Rightarrow$) Let $M$ be the maximal ideal of $R$. Similar to the proof of Lemma \ref{ch2}, we see that $\inv(R)\subseteq 1+M$. Moreover, as $\G(R)$ is $4$-regular and $-1\in \inv(R)$, we may assume that
$$\inv(R)=\{1,-1,1+x,-1+x\},$$
where $x\in M\setminus \{0,2\}$. Since $1=(1+x)^2=1+2x+x^2$, we infer that $x^2=2x$.

Note that each member of $A=\{0,x,2,x+2\}$ is adjacent to every member of $1+A=\{1,1+x,-1,-1+x\}$. Since  $\G(R)$ is connected and $4$-regular, it can be concluded that $R=A\cup (1+A)$ and $|R|=8$. It is clear that $M=A$ and $\G(R)\simeq K_{4,4}$. 

On the other hand, we have $2x=x^2\in M=\{0,2,x,2+x\}$. It is not hard to see that if $2x$ is either $2$, $x$ or $2+x$, we get a contradiction. Hence $x^2=2x=0$. Consequently, 
$$R=\{a+bx \mid a=0,1,2,3, \,b=0,1\},$$
 from which we infer that $R$ is isomorphic to 
 $\mathbb{Z}_4[x]/\langle x^2,2x \rangle$.
 
 ($\Leftarrow$)
 Since $\G(R)\simeq K_{4,4}$ for $R=\mathbb{Z}_4[x]/\langle x^2,2x \rangle $, the result follows.
\end{proof}

For a better illustration of the proof of the next lemma, we present the following example.

\begin{example}\label{ex-zn}
The involutory Cayley graphs of the rings $\mathbb{Z}_8$ and $\mathbb{Z}_{16}$ have genus $1$. To see this, note that they are not planar, as Theorem \ref{planar} indicates. Moreover, for a toroidal embedding, consider the following figures. For $\mathbb{Z}_8$, the vertices $0$ and $4$ of the inner rhombus can be connected to the outer rhombus by adding a handle. 

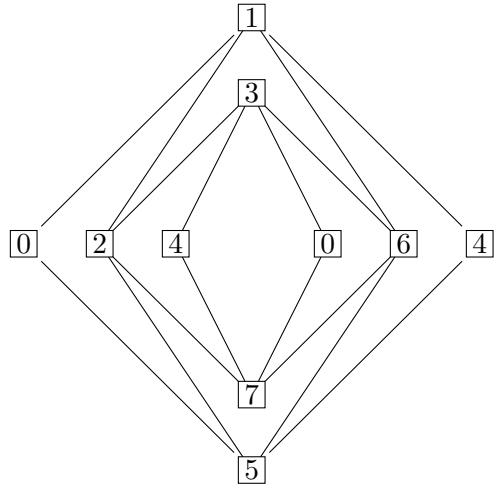
\begin{center}
	\begin{tikzpicture}
	\node (v3) at (-0.5,3.5) {5};
	\node (v1) at (-0.5,9.5) {1};
	\node (v4) at (2.5,6.5) {4};
	\node (v2) at (-3.5,6.5) {0};
	\draw  (v1) edge (v2);
	\draw  (v2) edge (v3);
	\draw  (v3) edge (v4);
	\draw  (v4) edge (v1);
	\node [draw,outer sep=0,inner sep=1,minimum size=10] (v6) at (-0.5,9.5) {};
	\node [draw,outer sep=0,inner sep=1,minimum size=10] at (-3.5,6.5) {};
	\node [draw,outer sep=0,inner sep=1,minimum size=10] (v7) at (-0.5,3.5) {};
	\node [draw,outer sep=0,inner sep=1,minimum size=10] at (2.5,6.5) {};
	\node [draw,outer sep=0,inner sep=1,minimum size=10] (v5) at (-2.5,6.5) {2};
	\node [draw,outer sep=0,inner sep=1,minimum size=10] (v8) at (1.5,6.5) {6};
	\draw  (v5) edge (v6);
	\draw  (v5) edge (v7);
	\draw  (v7) edge (v8);
	\draw  (v8) edge (v6);
	\node [draw,outer sep=0,inner sep=1,minimum size=10] (v9) at (-0.5,8.5) {3};
	\node [draw,outer sep=0,inner sep=1,minimum size=10] (v10) at (-0.5,4.5) {7};
	\draw  (v9) edge (v5);
	\draw  (v5) edge (v10);
	\draw  (v10) edge (v8);
	\draw  (v8) edge (v9);
	\node [draw,outer sep=0,inner sep=1,minimum size=10] (v11) at (-1.5,6.5) {4};
	\node [draw,outer sep=0,inner sep=1,minimum size=10] (v12) at (0.5,6.5) {0};
	\draw  (v9) edge (v11);
	\draw  (v11) edge (v10);
	\draw  (v10) edge (v12);
	\draw  (v12) edge (v9);
\end{tikzpicture}
\end{center}

\noindent
Likewise, for $\mathbb{Z}_{16}$, the vertices $0$ and $8$ can be connected to the outer rhombus via a handle, as in the case of $\mathbb{Z}_8$.

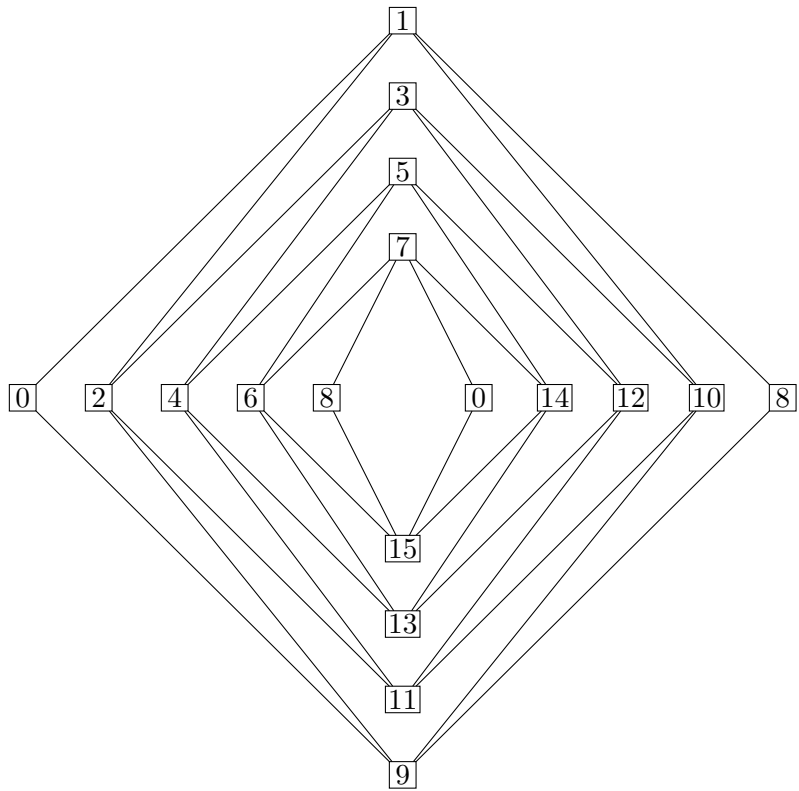
\begin{center}
	\begin{tikzpicture}
	
	\node (v3) at (-0.5,3.5) {};
	\node (v1) at (-0.5,9.5) {};
	\node (v4) at (2.5,6.5) {};
	\node (v2) at (-3.5,6.5) {};
	
	\node [draw,outer sep=0,inner sep=1,minimum size=10] (v6) at (-0.5,9.5) {5};
	\node [draw,outer sep=0,inner sep=1,minimum size=10] (v21) at (-3.5,6.5) {4};
	\node [draw,outer sep=0,inner sep=1,minimum size=10] (v7) at (-0.5,3.5) {13};
	\node [draw,outer sep=0,inner sep=1,minimum size=10] (v22) at (2.5,6.5) {12};
	\node [draw,outer sep=0,inner sep=1,minimum size=10] (v5) at (-2.5,6.5) {6};
	\node [draw,outer sep=0,inner sep=1,minimum size=10] (v8) at (1.5,6.5) {14};
	\node [draw,outer sep=0,inner sep=1,minimum size=10] (v9) at (-0.5,8.5) {7};
	\node [draw,outer sep=0,inner sep=1,minimum size=10] (v10) at (-0.5,4.5) {15};
	\node [draw,outer sep=0,inner sep=1,minimum size=10] (v11) at (-1.5,6.5) {8};
	\node [draw,outer sep=0,inner sep=1,minimum size=10] (v12) at (0.5,6.5) {0};
	\draw  (v9) edge (v11);
	\draw  (v11) edge (v10);
	\draw  (v10) edge (v12);
	\draw  (v12) edge (v9);
	\node [draw,outer sep=0,inner sep=1,minimum size=10] (v17) at (-4.5,6.5) {2};
	\node [draw,outer sep=0,inner sep=1,minimum size=10] (v18) at (3.5,6.5) {10};
	\node [draw,outer sep=0,inner sep=1,minimum size=10] (v20) at (-0.5,2.5) {11};
	\node [draw,outer sep=0,inner sep=1,minimum size=10] (v19) at (-0.5,10.5) {3};
	\node [draw,outer sep=0,inner sep=1,minimum size=10] (v14) at (-5.5,6.5) {0};
	\node [draw,outer sep=0,inner sep=1,minimum size=10] (v13) at (-0.5,11.5) {1};
	\node [draw,outer sep=0,inner sep=1,minimum size=10] (v16) at (4.5,6.5) {8};
	\node [draw,outer sep=0,inner sep=1,minimum size=10] (v15) at (-0.5,1.5) {9};
	\draw  (v13) edge (v14);
	\draw  (v14) edge (v15);
	\draw  (v15) edge (v16);
	\draw  (v16) edge (v13);
	\draw  (v13) edge (v17);
	\draw  (v17) edge (v15);
	\draw  (v15) edge (v18);
	\draw  (v18) edge (v13);
	\draw  (v19) edge (v17);
	\draw  (v17) edge (v20);
	\draw  (v20) edge (v18);
	\draw  (v18) edge (v19);
	\draw  (v19) edge (v21);
	\draw  (v21) edge (v20);
	\draw  (v20) edge (v22);
	\draw  (v22) edge (v19);
	\draw  (v6) edge (v21);
	\draw  (v21) edge (v7);
	\draw  (v7) edge (v22);
	\draw  (v22) edge (v6);
	\draw  (v6) edge (v5);
	\draw  (v5) edge (v7);
	\draw  (v7) edge (v8);
	\draw  (v8) edge (v6);
	\draw  (v5) edge (v9);
	\draw  (v5) edge (v10);
	\draw  (v10) edge (v8);
	\draw  (v8) edge (v9);
\end{tikzpicture}
\end{center}
\end{example}

\begin{lemma}\label{ch2n}
Let $R$ be a local ring with $\chr(R)=2^n$ for some $n\geq 3$. The following statements are equivalent.
\begin{enumerate}
\item\label{ch2n1} $\G(R)$ is of genus $1$;

\item\label{ch2n2} $\G(R)$ is connected and $4$-regular;

\item\label{ch2n3} $R$ is isomorphic to $\mathbb{Z}_{2^n}$.
\end{enumerate}
\end{lemma}
\begin{proof}
\pref{ch2n1} $\Rightarrow$ \pref{ch2n2} This follows from Lemma \ref{nec}.

\pref{ch2n2} $\Rightarrow$ \pref{ch2n3}
Since $\chr(R)=2^n$, the ring $R$ contains a copy of $\mathbb{Z}_{2^n}$. Therefore, $\G(\mathbb{Z}_{2^n})$ is a subgraph of $\G(R)$. Using \cite[Example 2.1]{inv}, we observe that  $\G(\mathbb{Z}_{2^n})$ is a connected  $4$-regular graph. Thus, we infer that $R$ is isomorphic to $\mathbb{Z}_{2^n}$, as desired.

\pref{ch2n3} $\Rightarrow$ \pref{ch2n1}
By Theorem \ref{planar}, we know that $\G(\mathbb{Z}_{2^n})$  is not planar.
Similar to Example \ref{ex-zn}, consider the following configuration for $\G(\mathbb{Z}_{2^{n}})$. For each $0\leq i \leq 2^{n-1} $, there exists a rhombus
$$ i \sim (i + 1) \sim (2^{n-1}+i) \sim (2^{n-1}+i+1) $$
and every element of $\mathbb{Z}_{2^n}$ appears in exactly two such rhombi. We begin with $i=0$ for the outer rhombus and increment  $i$ to move inward until reaching $i=2^{n-1}$. This configuration yields a planar embedding except for the two vertices $0$ and $2^{n-1}$, which lie on both the outermost and innermost rhombi.  By adding a handle (i.e., a toroidal bridge), we can connect these two rhombi, resulting in an embedding on a torus. This completes the proof.
\end{proof}

We are now in a position to state the main theorem, which classifies all rings $R$ such that $\G(R)$ has genus $1$.

\begin{thm}\label{main}
Let $R$ be a ring. Then, the involutory Cayley graph of $R$ is of genus $1$ if and only if one of the following conditions hold.
\begin{enumerate}
\item\label{main 1}
$R$ is isomorphic to $\mathbb{Z}_2[x,y]/\langle x^2,xy,y^2\rangle$, $\mathbb{Z}_4[x]/\langle x^2,2x\rangle$ or $\mathbb{Z}_{2^n}$ for some $n\geq 3$;

\item\label{main 2}
$R$ is isomorphic to $\mathbb{Z}_{p^n}\times \mathbb{Z}_{q^m}$ ,  $\mathbb{Z}_{p^n}\times \mathbb{Z}_{q^m}\times \mathbb{Z}_2$ , $\mathbb{Z}_{p^n}\times \mathbb{Z}_4$ or  $\mathbb{Z}_{p^n}\times \mathbb{Z}_2[x]/\langle x^2 \rangle $
for some odd prime numbers $p,q$ and positive integers $n,m$.
\end{enumerate}

In particular, $\G(R)$ is of genus $1$ if and only if it is connected and $4$-regular.
\end{thm}
\begin{proof}
($\Rightarrow$) Since $\G(R)$ is of genus $1$, applying Lemma \ref{nec}, it is connected and $4$-regular. If $R$ is not local, by Lemma \ref{nonl}, we obtain the condition \pref{main 2}. Moreover, if $R$ is local, then by Theorem \ref{imp} \pref{imp 2}, we infer that $R$ is of even size. In this case, combining Lemmas \ref{ch2}, \ref{ch4} and \ref{ch2n}, we get the condition \pref{main 1}, obtaining the result.

($\Leftarrow$)
Keeping in mind Theorem \ref{imp}, it can be seen that
\begin{align*}
	\G(\mathbb{Z}_{p^n}\times \mathbb{Z}_{q^m})&\simeq C_{p^n}\otimes C_{q^m},\\
	\G(\mathbb{Z}_{p^n}\times \mathbb{Z}_4)&\simeq C_{p^n}\otimes C_4,\\
	\G(\mathbb{Z}_{p^n}\times \mathbb{Z}_2[x]/\langle x^2 \rangle )&\simeq C_{p^n}\otimes C_4 ,
\end{align*}
and 
$$ \G(\mathbb{Z}_{p^n}\times \mathbb{Z}_{q^m}\times \mathbb{Z}_2)\simeq C_{p^n}\otimes C_{q^m}\otimes K_2 \simeq C_{p^n}\otimes C_{2q^m}.$$
The genus of the direct product of two cycles is known to be $2$ if both cycles  have even length and $1$ otherwise (see \cite[Corollary 3]{conj}). Hence, all four of the above graphs have genus $1$.

Moreover,
one could obeserve that if $R$ is isomorphic to $\mathbb{Z}_2[x,y]/ \langle x^2,xy,y^2\rangle $
or
$\mathbb{Z}_4[x]/\langle x^2,2x\rangle$, then 
$\G(R)\simeq K_{4,4}$. The genus of complete bipartite graphs can be calculated by the well-known formula
$$ g(K_{m,n})= \left\lceil\dfrac{(m-2)(n-2)}{4}\right\rceil$$
for $m,n\geq 2$
(see, for example, \cite[Theorem 6.37]{white}). Therefore, the genus of $K_{4,4}$ is equal to $1$, as required. Also, if $R\simeq \mathbb{Z}_{2^n}$, the result follows from Lemma \ref{ch2n} and we are done.

The "In particular" statement  follows from Lemmas \ref{nec}, \ref{nonl}, \ref{ch2}, \ref{ch4} and \ref{ch2n}.
\end{proof}

To conclude, we propose some problems for further research.
Recalling Proposition \ref{conn}, we have a characterization of connected involutory Cayley graphs. However, to obtain a full classification, one must classify all local rings of even size whose involutory Cayley graphs are connected. In our study in \cite{inv} and during the preparation of this paper, we were unable to resolve this problem. Therefore, one possible direction is to classify all connected involutory Cayley graphs.

Another direction is to investigate other important graph-theoretic properties, such as spectral properties. Additionally, it would be interesting to find necessary and sufficient conditions for two rings $R$ and $S$ such that  $\G(R)\simeq \G(S)$.

\end{document}